\documentclass[11pt, a4paper, fleqn, final]{article}
\usepackage[utf8]{inputenc}
\usepackage[T1]{fontenc}
\usepackage[english]{babel}
\usepackage{amsmath, amsthm, amssymb, mathtools, dsfont}
\usepackage{csquotes}
\usepackage[backend=bibtex,sortcites=true,isbn=false,maxnames=5,mincrossrefs=1,parentracker=true,sorting=nyt]{biblatex}
\usepackage{libertine}

\usepackage{ifdraft}
\usepackage[draft]{showkeys}
\usepackage[obeyDraft]{todonotes}

\allowdisplaybreaks[4]
\newtheorem{theo}{Theorem}
\newtheorem{coro}[theo]{Corollary}
\newtheorem{prop}[theo]{Proposition}
\newtheorem{lemm}[theo]{Lemma}

\theoremstyle{definition}
\newtheorem{defi}[theo]{Definition}

\newcommand{\eqspace}{\ensuremath{\mathrel{\phantom{=}}}}
\newcommand{\vig}[2]{#1[#2]} 
\newcommand{\edel}{-} 
\newcommand{\vdel}{\hspace{-0.5pt} \ominus} 
\newcommand{\vsdel}{\hspace{-0.5pt} \ominus} 
\newcommand{\vhid}{{\sim}} 
\newcommand{\vshid}{{\sim}} 
\newcommand{\esub}{{\times}} 

\DeclarePairedDelimiter\abs{\lvert}{\rvert}

\newcommand{\sumsr}[2]{\smashoperator[r]{\sum\limits_{#1}} #2}

\makeatletter
\def\moverlay{\mathpalette\mov@rlay}
\def\mov@rlay#1#2{\leavevmode\vtop{%
   \baselineskip\z@skip \lineskiplimit-\maxdimen
   \ialign{\hfil$\m@th#1##$\hfil\cr#2\crcr}}}
\newcommand{\charfusion}[3][\mathord]{
    #1{\ifx#1\mathop\vphantom{#2}\fi
        \mathpalette\mov@rlay{#2\cr#3}
      }
    \ifx#1\mathop\expandafter\displaylimits\fi}
\makeatother
\newcommand{\cdotcup}{\charfusion[\mathbin]{\cup}{\cdot}}

\title{A survey on recurrence relations for the independence polynomial of hypergraphs\thanks{This works was supported by the National Natural Science Foundation of China.}}
\author{Martin Trinks \\[+0.8ex] 
Center for Combinatorics, 
Nankai University, 
Tianjin 300071, China \\[+0.8ex] 
\small \tt martin.trinks@googlemail.com}

\date{}

\begin{document}

\maketitle

\begin{abstract}
The independence polynomial of a hypergraph is the generating function for its independent (vertex) sets with respect to their cardinality. This article aims to discuss several recurrence relations for the independence polynomial using some vertex and edge operations. Further, an extension of the well-known recurrence relation for simple graphs to hypergraphs is proven and other novel recurrence relations are also discussed.
\\[+1.6ex]
\\[+1.6ex]
\end{abstract}

\section{Introduction}
\label{sec:introduction}

A vertex subset of a simple graph $G$ is called an \emph{independent (vertex) set} if it does not include two adjacent vertices. The generating function of the number of independent sets of $G$ with respect to their cardinalities is known as the \emph{independence polynomial} $I(G, x)$ \cite{gutman1983, levit2005}. That is $I(G, x) = \sum_{i}{a_i(G) \: x^{i}}$, where $a_i(G)$ is the number of independent sets of $G$ with exaclty $i$ vertices.

The independence polynomial satisfies a recurrence relation \cite[Proposition~7]{gutman1983} with respect to the deletion of a vertex $v$, denoted by $\vdel v$, and the deletion of $v$ and its neighboring vertices, denoted by $\vsdel N[v]$. These operations give rise to the \emph{standard recurrence relation}:
\begin{align}
I(G, x) = I(G_{\vdel v}, x) + x \cdot I(G_{\vsdel N[v]}, x).
\end{align}

Contrary to the recurrence relations of some other well-known graph polynomials, such as the chromatic polynomial, the matching polynomial and the Potts model, whose recurrence relations established for simple graphs are also valid for hypergraphs, the recurrence relation of the independence polynomial of simple graphs does not extend to hypergraphs, in general. Thus, the purpose of this work is to demonstrate several recurrence relations for generating the independence polynomial of hypergraphs via both vertex and edge operations.

Further, we aim to contribute to the long-term goal of ``a general theory of graph polynomials'', a term due to \textcite{makowsky2008}. Thus, we give an in-depth study on the recurrence relations of independence polynomials, and thereby demonstrate how to use several different graph operations to arrive at different recurrence relations for the same polynomial.

The paper is organized as follows: After giving all necessary definitions and notation in Section \ref{sec:preliminaries}, we extend well-known recurrence relations of simple graphs to hypergraphs in Section \ref{sec:known_rec_rel}. Thereafter, in Section \ref{sec:new_rec_rel}, novel recurrence relations involving hypergraphs are introduced. In Section \ref{sec:vertex_cover_polynomial} a graph polynomial equivalent to the independence polynomial, the vertex cover polynomial, is discussed and its relation with the aforementioned recurrences are established. We complete this work with a short discussion in Section \ref{sec:open_problems}.
\section{Preliminaries}
\label{sec:preliminaries}

\begin{defi}
A \emph{hypergraph} $G = (V, E)$  is an ordered pair of a set of vertices,
the vertex set $V$, and a multiset of (hyper)edges, the edge set E, such that each edge is a non-empty subset of the vertex set, i.e.\ $e \subseteq V$ for all $e \in E$.
\end{defi}

Consequently, a hypergraph $G = (V, E)$ is a \emph{graph} if each edge is a set of at most two vertices, i.e.\ $\abs{e} \leq 2$ for all $e \in E$, and it is even more a \emph{simple graph} if each edge has exactly two vertices and the edge set is a set, i.e.\ $E \subseteq \binom{V}{2}$.

By $V(G)$, $n(G) = \abs{V(G)}$, $E(G)$ and $m(G) = \abs{E(G)}$ we denote the vertex set, number of vertices, edge set and number of edges of a hypergraph $G$, respectively. An edge is denoted as \emph{loop} if it has exactly one incident vertex.

\begin{defi}
Let $G = (V, E)$ be a hypergraph. A vertex subset $W \subseteq V$ is an \emph{independent set} in $G$ if $\forall e \in E \colon e \nsubseteq W$. \end{defi}

Actually, the definition of independent sets in simple graphs could also be  generalized to hypergraphs by requiring that from each edge at most one vertex is in an independent set. (Instead of requiring that not all vertices of any edge are in an independent set.) But this would be equivalently to replacing a (hyper)edge by edges between every two of its vertices. An analogous situation occured for the chromatic polynomial \cite[Footnote 2 on page 281]{stanley1998}. With the chosen definition, several relations between independent sets and proper colorings remain, for instance that a color class of a proper coloring is an independent set.

\begin{defi}
Let $G = (V, E)$ be a hypergraph. The \emph{independence polynomial} $I(G, x)$ is defined as
\begin{align*}
I(G, x) = \sumsr{\substack{W \subseteq V \\ W \text{ is independent in }G}}{x^{\abs{W}}}.
\end{align*}
\end{defi}

The first formal definition of the independence polynomial appears to be due to \textcite{gutman1983}. In fact, such a polynomial was earlier observed in statistical mechanics, see \cite{heilmann1972} and the references therein. The independence polynomial is well-studied, see the surveys \cite{hoede1994, levit2005}. In the literature, it is also known as \emph{independent set polynomial} \cite{hoede1994} and \emph{stable set polynomial} \cite{stanley1998}.

In an edgeless graph on $n$ vertices, denoted by $E_n$, every vertex subset is an independent set, hence $I(E_n, x) = (1 + x)^n$. From this initial value and by applying the standard recurrence relation, the independence polynomial of an arbitrary simple graph can be calculated. This procedure is often shortened by employing the multiplicativity in components, that is
\begin{align}
I(G_1 \cdotcup G_2, x) = I(G_1, x) \cdot I(G_2, x),
\end{align}
where $G_1 \cdotcup G_2$ is the disjoint union of the graphs $G_1$ and $G_2$.

For the sake of convenience we make use of a \emph{restricted independence polynomial} $I_S(G, x)$, where $S$ is a statement restricting the independent sets enumerated. (We omit $W$ in the statements.) For example, the independence polynomial counting all the independent sets of a graph $G$ including some subset $U$ and not including the vertex $v$ are, respectively, 
\begin{alignat*}{2}
& I_{U \subseteq}(G, x) &&= \sumsr{\substack{U \subseteq W \subseteq V \\ W \text{ is independent in }G}}{x^{\abs{W}}}, \qquad 
I_{v \notin }(G, x) = \sumsr{\substack{v \notin W \subseteq V \\ W \text{ is independent in }G}}{x^{\abs{W}}}.
\end{alignat*}

For a hypergraph $G=(V, E)$ with a vertex subset $W \subseteq V$, $N_G(W)$ denotes the \emph{open neighborhood} of $W$, which is the set of vertices adjacent to some vertex of $W$ in $G$: $N_G(W) = \{v \mid \{v, w\} \subseteq e \in E, v \in V, w \in W\}$. The \emph{closed neighborhood} $N_G[W]$ of $W$ is the union of the open neighborhood and the set $W$ itself, $N_G[W] = N_G(W) \cup W$. If $W = \{v\}$, we write $N_G(v)$ and $N_G[v]$ instead of $N_G(\{v\})$ and $N_G[\{v\}]$, respectively.

Given a hypergraph $G = (V, E)$, a vertex subset $W \subseteq V$ and an edge $e \in E$, we define the following graph operations:
\begin{itemize}
\item $\vsdel W$: \emph{deletion} of the vertices $v \in W$, i.e.\ the vertices $v$ and their incident edges are removed,
\item $\vshid W$: \emph{hiding} of the vertices $v \in W$, i.e.\ the vertices $v$ are removed in the vertex set and in their incident edges,
\item $-e$: \emph{deletion} of edge $e$, i.e.\ edge $e$ is removed,
\item $/e$: \emph{contraction} of edge $e$, i.e.\ edge $e$ is removed and its incident vertices are unified,
\item $\esub e$: \emph{subdivision} of edge $e$, i.e.\ edge $e$ is removed, and a vertex $d$ and edges $\{v, d\}$ for every vertex $v \in e$ are added.
\end{itemize}

We denote the graphs arising from these graph operations by the original graph with the graph operation added in the subscript, for example $G_{\vsdel W}$, $G_{\vshid W}$, $G_{-e}$, $G_{/e}$, $G_{\esub e}$. Again, for the sake of brevity, we write $G_{-v}$ and $G_{\vhid v}$ if $W = \{v\}$.

Let $G$ and $H$ be two hypergraphs. $H$ is a \emph{subgraph} of $G$, if $V(H) \subseteq V(G)$ and $E(H) \subseteq E(G)$. The (\emph{vertex}-)\emph{induced subgraph} $G[W]$ for some vertex subset $W \subseteq V(G)$ is the hypergraph with vertex set $W$ and all edges of $E(G)$ that are a subset of $W$, i.e.\ $\vig{G}{H} = (W, \{e \in E(G) \mid e \subseteq W \})$. In other words, a subgraph arises by vertex and edge deletion and an induced subgraph arises by vertex deletion only.
\section{Modification of known recurrence relations}
\label{sec:known_rec_rel}

There are already several recurrence relations known for the independence polynomial. The most widely used is the standard recurrence relation given above.

\begin{prop}[Proposition 7 in \cite{gutman1983}] \label{prop:rec_rel_sg_v}
Let $G = (V, E)$ be a simple graph and $v \in V$. The independence polynomial satisfies 
\begin{align}
I(G, x) = I(G_{\vdel v}) + x \cdot I(G_{\vsdel N_G[v]}).
\end{align}
\end{prop}

In this recurrence relation, the first term counts exactly those independent sets of $G$ not including the vertex $v$, whereas the second term counts exactly those independent sets including the vertex $v$ (and therefore none of its neighbors). This recurrence relation can be generalized to any vertex subset in two different ways as stated in the next two theorems.

\begin{theo}[Theorem 3.7 in \cite{hoede1994}] \label{theo:rec_rel_sg_vs_1}
Let $G = (V, E)$ be a simple graph and $U \subseteq V$ a vertex subset. Then
\begin{align}
& I(G, x) = I(G_{\vsdel U}, x) + \sumsr{\substack{\emptyset \subset W \subseteq U \\ W \text{ is independent in }G}}{(-1)^{\abs{W}+1} \: x^{\abs{W}} \cdot I(G_{\vsdel N_{G}[W]}, x)}. 
\end{align}
\end{theo}

In this equality, the summation is over the vertex subsets of $U$ which are at least in the independent sets of $G$. Changing the summation to the vertex subsets exactly in the independent sets of $G$, the expression simplifies.

\begin{theo}[\footnote{In \cite{trinks2013} the present author mentioned a special case of this theorem (as Theorem 2) and misleadingly referred it to the reference of the Theorem \ref{theo:rec_rel_sg_vs_1} above.}] \label{theo:rec_rel_sg_vs}
Let $G = (V, E)$ be a simple graph and $U \subseteq V$ a vertex subset. Then
\begin{align}
& I(G, x) = \sumsr{\substack{W \subseteq U \\ W \text{ is independent in }G}}{x^{\abs{W}} \cdot I(G_{\vsdel U \vsdel N_{G}(W)}, x)}.
\end{align}
\end{theo}

\begin{proof}
We prove the statement by induction on the number of vertices in $U$. As basic step we consider the case $|U| = 1$, say $U = \{u\}$. From the standard recurrence relation (Proposition \ref{prop:rec_rel_sg_v}) we get
\begin{align*}
I(G, x) &= I(G_{\vdel u}, x) + x \cdot I(G_{\vdel u \vsdel N_{G}(u)}, x), \\
&= \sumsr{\substack{W \subseteq U = \{u\} \\ W \text{ is independent in }G}}{I(G_{\vsdel U \vsdel N_{G}(W)}, x)}.
\end{align*}
As induction hypothesis we assume that the statement holds for all sets $U$ with $|U| \leq n$. Let $U' = U \cup \{u\}$ with $\abs{U} = n$. Applying the standard recurrence relation and the induction hypothesis we obtain
\begin{align*}
I(G, x) 
&= I(G_{\vdel u}, x) + I(G_{\vdel u \vsdel N_{G}(u)}) \\
&= \sum_{\mathrlap{\substack{W \subseteq U \\ W \text{ is independent in }G_{\vdel u}}}}{I(G_{\vdel u \vsdel U \vsdel N_{G_{\vdel u}}(W)}, x)} + \sum_{\mathrlap{\substack{W \subseteq U \setminus N_{G}(u) \\ W \text{ is independent in }G_{\vdel u \vsdel N_G(u)}}}}{I(G_{\vdel u \vsdel N_{G}(u) \vsdel U \vsdel N_{G_{\vdel u \vsdel N_G(u)}}(W)}, x)} \\
&= \sum_{\mathrlap{\substack{W \subseteq U \\ W \text{ is independent in }G_{\vdel u}}}}{I(G_{\vdel u \vsdel U \vsdel N_{G}(W)}, x)} + \sum_{\mathrlap{\substack{W \subseteq U \setminus N_{G}(u) \\ W \text{ is independent in }G_{\vdel u \vsdel N_G(u)}}}}{I(G_{\vdel u \vsdel N_{G}(u) \vsdel U \vsdel N_{G}(W)}, x)} \\
&= \sum_{\mathrlap{\substack{W' \subseteq U' \\ W' \text{ is independent in }G \\ u \notin W'}}}{I(G_{\vsdel U' \vsdel N_{G}(W')}, x)} + \sum_{\mathrlap{\substack{W' \subseteq U' \\ W' \text{ is independent in} G \\ u \in W'}}}{I(G_{\vsdel U' \vsdel N_{G}(W')}, x)} \\  
&= \sum_{\mathrlap{\substack{W' \subseteq U' \\ W' \text{ is independent in }G}}}{I(G_{\vsdel U' \vsdel N_{G}(W')}, x)},
\end{align*}
where the third equality holds because 
$G_{\vsdel A \vsdel N_{G_{\vsdel A}}(B)} = G_{\vsdel A \vsdel N_G(B)}$.
Thus, the statement also holds when $|U'| = n+1$, which proves the theorem.
\end{proof}

In the special case were $U$ is a clique, the index of the summation reduces to the vertices of $U$.

\begin{coro}[Corollary 3.8 in \cite{hoede1994}]
Let $G = (V, E)$ be a simple graph and $U \subseteq V$ a vertex subset that induces a complete subgraph. Then
\begin{align}
I(G, x) = I(G_{\vsdel U}, x) + x \cdot \sumsr{u \in U}{I(G_{\vsdel N[v]}, x)}.
\end{align}
\end{coro}

As previously mentioned in the introduction, the standard recurrence relation as given in Proposition \ref{prop:rec_rel_sg_v} is not valid for hypergraphs. For example, observe a hypergraph with three vertices connected by the sole edge. In fact, the standard recurrence relation is not true for graphs with loops, even in the case of a single vertex with a loop.

The reason is that in such situation the second term of the sum, $I(G_{\vsdel N[v]}, x)$, does not count the number of independent sets including $v$ correctly. In the case of a loop incident to $v$, $I(G_{\vsdel N[v]}, x)$ counts too much: At least the empty set is counted as an independent set of $G_{\vsdel N[v]}$, which corresponds to the independent set $\{v\}$ of $G$, but there are no independent sets including $v$. On the other hand, if $v$ is a vertex of a (hyper)edge with more than two vertices, then $I(G_{\vsdel N[v]}, x)$ counts too little: By deleting the neighborhood of $v$, there are also some vertices deleted, that may be in an independent set together with $v$.

Instead of $G_{\vsdel N[v]}$, the hypergraph $G_{\vhid v}$ can be observed to count the number of independent vertex set including $v$. This approach also works for vertex subsets.

\begin{lemm} \label{lemm:independent_sets_including}
Let $G = (V, E)$ be a hypergraph with independent sets $U, \{v\} \subseteq V$. The restricted independence polynomial satisfies
\begin{alignat}{2}
& I_{U \subseteq }(G, x) 
&&= x^{\abs{U}} \cdot I(G_{\vshid U}, x), \quad \text{and particularly} \\
& I_{v \in}(G, x) 
&&= x \cdot I(G_{\vhid v}, x).
\end{alignat}
\end{lemm}

\begin{proof}
The statement equals the claim that (for an independent set $U$) the independent sets of $G_{\vshid U}$, together with $U$, are exactly the independent sets of $G$ including $U$. Independent sets of a graph are exactly those vertex subsets not including any edge of that graph. Therefore, in each independent set of $G$ including $U$ at least one vertex (not in $U$) of each edge of $G$ is missing. But the same holds for the independent set of $G_{\vshid U}$ joined with $U$, because the edges of $G_{\vshid U}$ are subsets of the edges of $G$ (missing the vertices from $U$).
\end{proof}

Indeed, if there is an edge which is a subset of $W$, then $G_{\vshid W}$ is not defined, as we would end up with an empty edge. We give a further comment about this observation in Section \ref{sec:open_problems}.

\begin{theo} \label{theo:rec_rel_hg_v}
Let $G = (V, E)$ be a hypergraph and $v \in V$. The independence polynomial satisfies
\begin{align}
I(G, x) = 
\begin{cases}
I(G_{\vdel v}, x) + x \cdot I(G_{\vhid v}, x) & \text{if } \{v\} \notin E, \\
I(G_{\vdel v}, x) & \text{else.}
\end{cases}
\end{align}
\end{theo}

\begin{proof}
The independent set of $G$ not including $v$ are the independent edges of $G_{\vdel v}$. If there is a loop incident to $v$, i.e.\ $\{v\} \in E$, then there is no independent set in $G$ including $v$. Thus, all independent sets of $G$ are counted by $I(G_{\vdel v}, x)$. Otherwise, if there are no loops incident to $v$, then by Lemma \ref{lemm:independent_sets_including} the independent sets of $G$ including $v$ are exactly those of $G_{\vhid v}$ joined with $\{v\}$.
\end{proof}

Indeed, this theorem is a generalization of the standard recurrence relation in the case of simple graphs: Applying $\vhid v$, loops arise for all neighbors of $v$. Consequently, these vertices can be deleted and we end up with the graph $G_{\vsdel N[v]}$.

Analogous to Theorem \ref{theo:rec_rel_sg_vs} for simple graphs, the theorem above can be generalized to vertex subsets.

\begin{theo} \label{theo:rec_rel_hg_vs}
Let $G = (V, E)$ be a hypergraph and $U \subseteq V$ a vertex subset. The independence polynomial satisfies
\begin{align}
I(G, x) = \sumsr{\substack{W \subseteq U \\ W \text{ is independent in }G}}{x^{\abs{W}} \cdot I(G_{\vshid W \vsdel U}, x)}.
\end{align}
\end{theo}

The proof of this statement is analogously to the case of simple graphs given in Theorem \ref{theo:rec_rel_sg_vs}. While the sum in the previous theorem concerns the vertices exactly in the independent sets, it is quite possible to consider the vertices at least in the independent sets, which gives a statement corresponding to Theorem \ref{theo:rec_rel_sg_vs_1}.

\begin{theo} \label{theo:rec_rel_hg_vs_1}
Let $G = (V, E)$ be a hypergraph and $U \subseteq V$ a vertex subset. The independence polynomial satisfies
\begin{align}
I(G, x) = I(G_{\vsdel U}, x) + \sumsr{\substack{\emptyset \subset W \subseteq U \\ W \text{ is independent in }G}}{(-1)^{\abs{W}+1} \: x^{\abs{W}} \cdot I(G_{\vshid W}, x)}
\end{align}
\end{theo}

\begin{proof}
The independent sets of $G$ not including any vertex of $U$ are enumerated by the first term, $I(G_{\vsdel U}, x)$. Thus, we are left with the enumeration of the independent set including at least one vertex of $U$, which can be calculated via the principle of inclusion-exclusion by $\sum_{\emptyset \subset W \subseteq U}{(-1)^{\abs{W}+1} \cdot I_{W \subseteq}(G, x)}$. Substituting $I_{W \subseteq}(G, x)$ for $x^{\abs{W}} \cdot I(G_{\vshid W}, x)$, following Lemma \ref{lemm:independent_sets_including}, we obtain the second term.
\end{proof}

We present another known recurrence relation for the independence polynomial of simple graphs with respect to an edge deletion.

\begin{prop}[Theorem 3.9 in \cite{hoede1994}] \label{prop:rec_rel_sg_e}
Let $G = (V, E)$ be a simple graph and $e = \{u, v\} \in E$. The independence polynomial satisfies 
\begin{align}
I(G, x) = I(G_{-e}, x) - x^2 \cdot I(G_{\vsdel N[e]}, x).
\end{align}
\end{prop}

Therein, the first term counts all independent sets of $G$, but additionally counts such including both vertices of $e$. But those are exactly subtracted by the second term.

As easily seen from the value $x^2$ related to the number of vertices of $e$, this recurrence relation is only valid for simple graphs. However, if there is no other edge that is a subset of the observed edge, the statement can be can be generalized to edges with more than two incident vertices using a similar idea.

\begin{theo} \label{theo:rec_rel_hg_e}
Let $G = (V, E)$ be a hypergraph and $e \in E$ an edge with $\abs{e} > 1$. The independence polynomial satisfies
\begin{align}
I(G, x) = 
\begin{cases}
I(G_{-e}, x) - x^{\abs{e}} \cdot I(G_{-e \vshid e}, x) & \text{if } \nexists f \in E \colon f \subseteq e, \\
I(G_{-e}, x) & \text{else.}
\end{cases}
\end{align}
\end{theo}

\begin{proof}
By the first term all independent sets of $G$ are counted, but additionally such independent set including all vertices of $e$.

If there exists an edge $f \in E$ with $f \subseteq e$, then independent set including all vertices of $e$ must also include all vertices of $f$ which contradict the assumption that these are independent sets. Consequently, in this case there are no independent set which are supersets of $e$.

Otherwise, i.e. if there is no such edge $f$, by Lemma \ref{lemm:independent_sets_including} the independent sets including the vertices of $e$ are enumerated by $x^{\abs{e}} \cdot I(G_{-e \vshid e}, x)$.
\end{proof}
\section{New recurrence relations}
\label{sec:new_rec_rel}

There are two other recurrence relations for the independence polynomial of a simple graph $G$ with an edge $e = \{u, v\}$ that appear to be unknown:
\begin{align}
I(G, x) &= I(G_{\vdel u}, x) + I(G_{\vdel v}, x) - I(G_{\vdel e}, x), \\
&= I(G_{-e}, x) - x \cdot I(G_{/e}, x) + x \cdot I(G_{\vsdel e}, x).
\end{align}

We derive generalizations of both equation for hypergraphs from a graph polynomial in four variables, the so-called generalized subgraph counting polynomial, which counts subgraphs with respect to some invariants.

\begin{defi}[Definition 5.1 in \cite{trinks2012c}]
\label{defi:gscp}
Let $G = (V, E)$ be a hypergraph. The \emph{generalized subgraph counting polynomial} $F(G, v, x, y, z)$ is defined as
\begin{align}
F(G, v, x, y, z) = \sumsr{H = (W, F) \subseteq G}{v^{\abs{W}} \, x^{k(H)} \, y^{\abs{F}} \, z^{\abs{E(\vig{G}{W})}}}.
\end{align}
\end{defi}

This graph polynomial generalizes several well-known graph polynomials, among others, the matching polynomial, the Potts model, the edge elimination polynomial and the subgraph component polynomial \cite[Section 5.2]{trinks2012c}. Furthermore, it satisfies a recurrence relation with respect to an edge deletion.

\begin{theo}[Theorem 5.2 in \cite{trinks2012c}]
\label{theo:gscp_rec_rel}
Let $G = (V, E)$ be a hypergraph with an edge $e \in E$. The generalized subgraph counting polynomial $F(G) = F(G, v, x, y, z)$ satisfies
\begin{align}
& F(G) = z \cdot F(G_{-e}) + v^{\abs{e}-1} \, y \, z \cdot F(G_{/e}) - v^{\abs{e}-1} \, y \, z \cdot F(G_{\vsdel e}) \notag \\
& \phantom{F(G)} \eqspace + (z-1) \cdot \sumsr{\emptyset \subset B \subseteq e}{(-1)^{\abs{B}} \cdot F(G_{\vsdel B})}, \\
& F(E_n) = (1 + v \, x)^n.
\end{align}
\end{theo}

The independence polynomial can be derived from the generalized subgraph counting polynomial through two different techniques. From these, two different recurrence relations follow, although both derivations are based on the fact that in hypergraphs a vertex subset $W$ is independent if the subgraph of $G$ induced by $W$ is edgeless.

\begin{theo} \label{theo:rel_gscp_i_1}
Let $G = (V, E)$ be a hypergraph. The independence polynomial is encoded in the generalized subgraph counting polynomial:
\begin{align}
I(G, x) = H(G, 1, x, 1, 0).
\end{align}
\end{theo}

\begin{proof}
The statement follows directly from the definition of both polynomials:
\begin{align*}
H(G, 1, x, 1, 0)
&= \sumsr{H = (W, F) \subseteq G}{1^{\abs{W}} \, x^{k(H)} \,  1^{\abs{F}} \, 0^{\abs{E(\vig{G}{W})}}} 
= \sumsr{\substack{H = (W, F) \subseteq G \\ \abs{E(\vig{G}{W})} = 0}}{x^{k(H)}} \\
&= \sumsr{\substack{W \subseteq V \\ W \text{ is independent in } G}}{x^{\abs{W}}} 
= I(G, x). \qedhere
\end{align*}
\end{proof}

\begin{coro} \label{coro:rec_rel_hg_e_1}
Let $G = (V, E)$ be a hypergraph with an edge $e \in E$. The independence polynomial satisfies
\begin{align}
I(G, x) = \sumsr{\emptyset \subset W \subseteq e}{(-1)^{\abs{W}+1} \cdot I(G_{\vsdel W}, x)}.
\end{align}
\end{coro}

We give another combinatorial proof via an inclusion-exclusion argument.
\begin{proof}
The independent sets of $G$ are the independent sets of $G_{\edel e}$ not including all vertices of $e$, or, equivalently, missing at least one vertex of $e$. The independent sets of $G_{\edel e}$ missing at least the non-empty vertex set  $W \subseteq e$ are enumerated by $G_{-e \vsdel W} = G_{\vsdel W}$. Consequently, via the principle of inclusion-exclusion the statement follows.
\end{proof}

\begin{theo} \label{theo:rel_gscp_i_2}
Let $G = (V, E)$ be a hypergraph. The independence polynomial is encoded in the generalized subgraph counting polynomial:
\begin{align}
I(G, x) = H(G, x, 1, -1, 1).
\end{align}
\end{theo}

\begin{proof}
The statement follows directly from the definition of both polynomials:
\begin{align*}
H(G, x, 1, -1, 1)
&= \sumsr{H = (W, F) \subseteq G}{x^{\abs{W}} \, 1^{k(H)} \, (-1)^{\abs{F}} \, 1^{\abs{E(\vig{G}{W})}}} \\
&= \sumsr{H = (W, F) \subseteq G}{x^{\abs{W}} \, (-1)^{\abs{F}}} 
= \sumsr{W \subseteq V}{x^{\abs{W}} \sumsr{F \subseteq E(\vig{G}{W})}{(-1)^{\abs{F}}}} \\
&= \sumsr{W \subseteq V}{x^{\abs{W}} \, 0^{\abs{E(\vig{G}{W})}}} 
= \sumsr{\substack{W \subseteq V \\ W \text{ is independent in } G}}{x^{\abs{W}}} 
= I(G, x). \qedhere
\end{align*}
\end{proof}

\begin{coro} \label{coro:rec_rel_hg_e_2}
Let $G = (V, E)$ be a hypergraph with an edge $e \in E$. The independence polynomial satisfies
\begin{align}
I(G, x) = I(G_{-e}, x) - x^{\abs{e}-1} \cdot I(G_{/e}, x) + x^{\abs{e}-1} \cdot I(G_{\vsdel e}, x).
\end{align}
\end{coro}

Here, we present another combinatorial argument to support this statement.

\begin{proof}
By the independence polynomial of $G_{-e}$, $I(G_{-e}, x)$, all independent sets of $G$ are enumerated. However, there are some independent sets of $G_{-e}$ including all the vertices of $e$, which are not independent in $G$. To counterbalance this, we subtract $x^{\abs{e}-1} \cdot I(G_{/e}, x)$,
where the vertices incident $e$ are contracted to the vertex $w$. The independent sets counted by this term including $w$ correspond to the independent sets of $G_{-e}$ including all the vertices of $e$ and thus compensate the independent sets additionally counted by $I(G_{-e}, x)$. But at the same time, we also subtract the terms corresponding to independent sets of $G_{/e}$ not including $w$. But these sets are enumerated by $I(G_{/e -w}, x)$, which is equivalent to $I(G_{\vsdel e}, x)$. Thus, the subtracted excess terms are compensated by $x^{\abs{e}-1} \cdot I(G_{\vsdel e}, x)$.
\end{proof}

The subdivision of an edge is a well-known graph operation. It has been used by \textcite[pp.\ 222]{merrifield1989} in a recurrence relation for the number of independent sets $\sigma(G) = I(G, 1)$: For a simple graph $G = (V, E)$ with an edge $e$ it is true that
\begin{align}
\sigma(G) = \sigma(G_{\times e}) - \sigma(G_{/e}).
\end{align}

This result can be generalized simultaneously to the independence polynomial and to hypergraphs.

\begin{theo}
Let $G = (V, E)$ be a hypergraph with an edge $e \in E$. The independence polynomial $I(G, x)$ satisfies 
\begin{align}
I(G, x) = I(G_{\esub e}, x) - x^{\abs{e}-1} \cdot I(G_{/e}, x) + (x^{\abs{e}-1} - x) \cdot I(G_{\vsdel e}, x).
\end{align}
\end{theo}

\begin{proof}
The independent sets of $G_{\esub e}$ can be divided into two sets, those that include the newly inserted vertex $d$ and those that do not. The independent sets of the first kind (including $d$) cannot also include any of the vertices of $e$, as those vertices are linked by an edge with $d$. Therefore, $I_{d \in e}(G_{\esub e}) = x \cdot I(G_{\vsdel e})$.
The independent sets of the second kind (not including $d$) are the independent sets of $G_{\esub e - d} = G_{-e}$. Thus, we have
\[ I(G_{\esub e}) = x \cdot I(G_{\vsdel e}) + I(G_{-e}). \]
Substituting $I(G_{-e})$ in Corollary \ref{coro:rec_rel_hg_e_2}, the result follows from the equation above.
\end{proof}
\section{Vertex cover polynomial}
\label{sec:vertex_cover_polynomial}

We now discuss a graph polynomial that is strongly related to the independence polynomial and even equivalent to it, the vertex cover polynomial. This graph polynomial is the generating function for vertex covers and was introduced by \textcite{dong2002}. 

\begin{defi}
Let $G = (V, E)$ be a hypergraph. A vertex subset $W \subseteq V$ is a \emph{vertex cover} in $G$ if $\forall e \in E \colon e \cap W \neq \emptyset$.
\end{defi}
In other words, a vertex cover includes at least one vertex from each edge.

\begin{defi}
Let $G = (V, E)$ be a hypergraph. The \emph{vertex cover polynomial} $\Psi(G, x)$ is defined as
\begin{align}
\Psi(G, x) = \sumsr{\substack{W \subseteq V \\ W \text{ is a vertex cover in }G}}{x^{\abs{W}}}.
\end{align}
\end{defi}

Just as in the case of the independence polynomial, for edgeless graphs each vertex subset is a vertex cover and therefore $\Psi(E_n) = (1+x)^n$. 

While an independent set misses at least one vertex of each edge, a vertex cover includes at least one vertex of each edge. Therefore, the vertices not in an independent set form a vertex cover and vice versa.

\begin{prop}
Let $G = (V, E)$ be a hypergraph with a vertex subset $W \subseteq V$. $W$ is an independent set if and only if $V \setminus W$ is a vertex cover.
\end{prop}

From this proposition the relations between both graph polynomials follows directly.

\begin{prop} \label{prop:rel_i_psi}
Let $G = (V, E)$ be a hypergraph. The independence polynomial $I(G, x)$ and the vertex cover polynomial $\Psi(G, x)$ are equivalent graph polynomials related by
\begin{alignat}{2}
& I(G, x) &&= x^{n(G)} \cdot \Psi(G, x^{-1}), \\
& \Psi(G, x) &&= x^{n(G)} \cdot I(G, x^{-1}).
\end{alignat}
\end{prop}

Furthermore, the recurrence relations from both are related. This is a special case of a more general theorem relating recurrence relations of graph polynomials \cite[Theorem 3.18]{trinks2012c}.

\begin{theo} \label{theo:rec_rel_i_psi}
Let $G = (V, E)$ be a hypergraph. The recurrence relations of the independence polynomial $I(G, x)$ and the vertex cover polynomial $\Psi(G, x)$ are related to each other by:
\begin{align}
\Psi(G, x) &= \sum_{i}{a_i(x^{-1}) \: x^{n(G)-n(G_i)} \cdot \Psi(G_i, x)} \\
\intertext{if and only if}
I(G, x) &= \sum_{i}{a_i(x) \cdot I(G_i, x)},
\end{align}
where $a_i(x)$ are polynomials in $x$ and $G_i$ are hypergraphs (not necessarily arising from graph operations).
\end{theo}

\begin{proof}
The statement follows direct from the relations of both graph polynomials given in Propostion \ref{prop:rel_i_psi}:
\begin{align*}
\Psi(G, x) 
&= x^{n(G)} \cdot I(G, x^{-1}) \\
&= x^{n(G)} \cdot \sum_{i}{a_i(x^{-1}) \cdot I(G_i, x^{-1})} \\
&= \sum_{i}{a_i(x^{-1}) \: x^{n(G)} \: (x^{-1})^{n(G_i)} \cdot \Psi(G_i, x)} \\
&= \sum_{i}{a_i(x^{-1}) \: x^{n(G)-n(G_i)} \cdot \Psi(G_i, x)}. \qedhere
\end{align*}
\end{proof}

Thus, for each recurrence relation of the independence polynomial there is a corresponding one for the vertex cover polynomial.

\begin{coro}
Let $G = (V, E)$ be a hypergraph with a vertex $v \in V$ and an edge $e \in E$. The vertex cover polynomial $\Psi(G, x)$ satisfies
\begin{alignat}{2}
& \Psi(G, x)
&&= \begin{cases}
x \cdot \Psi(G_{\vdel v}, x) + \Psi(G_{\vhid v}, x) & \text{if } \{v\} \notin E, \\
x \cdot \Psi (G_{\vdel v}, x) & \text{else,}
\end{cases} \\
& &&= \begin{cases}
\Psi(G_{-e}, x) - \Psi(G_{-e \vshid e}, x) & \text{if } \nexists f \in E \colon f \subseteq e, \\
\Psi(G_{-e}, x) & \text{else.}
\end{cases} \\
& &&= \sumsr{\emptyset \subset B \subseteq e}{- (-x)^{\abs{B}} \cdot \Psi(G_{\vsdel B}, x)}, \\
& &&= \Psi(G_{-e}, x) - \Psi(G_{/e}, x) + x \cdot \Psi(G_{\vsdel e}, x). \label{eq:vc_rec_rel}
\end{alignat}
\end{coro}

\begin{proof}
The statements follow via Theorem \ref{theo:rec_rel_i_psi} from Theorem \ref{theo:rec_rel_hg_v}, Theorem \ref{theo:rec_rel_hg_e}, Corollary \ref{coro:rec_rel_hg_e_1} and Corollary \ref{coro:rec_rel_hg_e_2}, respectively.
\end{proof}

The recurrence relation corresponding to the standard recurrence relation of the independence polynomial is well-known \cite[Theorem 2.2]{dong2002}. (For a recurrence relation of $(-1)^{n(G)} \cdot \Psi(G, -x)$, see \cite[Equation (9)]{gutman1992}.) Equation \eqref{eq:vc_rec_rel} has already been derived in the case of graphs via other graph polynomials \cite[Corollary 30]{trinks2012}.

Recurrence relations via vertex subsets can be derived using the same approach. Here is an example.

\begin{coro}
Let $G = (V, E)$ be a hypergraph with a vertex subset $U \subseteq V$. The vertex cover polynomial satisfies
\begin{align}
\Psi(G, x) &= \sumsr{\substack{W \subseteq U \\ W \text{ is independent in }G}}{x^{\abs{U \setminus W}} \cdot \Psi(G_{\vshid W \vsdel U}, x)} \\
&= \sumsr{\substack{W \subseteq U \\ V \setminus W \text{ is a vertex cover in }G}}{x^{\abs{U \setminus W}} \cdot \Psi(G_{\vshid W \vsdel U}, x)}.
\end{align}
\end{coro}

\begin{proof}
The statements follow from Theorems \ref{theo:rec_rel_i_psi} and \ref{theo:rec_rel_hg_vs}.
\end{proof}
\section{Discussion}
\label{sec:open_problems}

In this survey, several recurrence relations for the independence polynomial of hypergraphs have been discussed. Therein many different vertex and edge operations have been used. However, this list of recurrence relations is not exhaustive, as one can involve other graph operations to generate new recurrence relations.

A part of the aforementioned recurrence relations requires a certain case distinction to prevent undefined graph operations, namely empty edges, that is, edges not incident to any vertex. Extending the definition of hypergraphs to include empty edges, this distinction can be omitted.

Let $G = (V, E)$ be such an \emph{extended hypergraph} also allowing empty edges, that means each edge $e$ is a (possibly empty) subset of the vertex set. An independent set is still defined as a vertex set not including any edge. Consequently, an extended hypergraph $G$ has no independent set if and only if it has an empty edge. Then the independence polynomial $I(G, x)$ satisfies
\begin{align}
I(G, x)
&= I(G_{\vdel v}, x) + x \cdot I(G_{\vhid v}, x), \\
&= I(G_{-e}, x) - x^{\abs{e}} \cdot I(G_{-e \vshid e}, x),
\end{align}
(without any exceptions).
Furthermore, the restriction in the summation in the equation for vertex subsets $U \subseteq V$ is not necessary:
\begin{align}
I(G, x)
= \sumsr{\substack{W \subseteq U \\ W \text{ is independent in }G}}{x^{\abs{W}} \cdot I(G_{\vshid W \vsdel U}, x)}
= \sumsr{\substack{W \subseteq U}}{x^{\abs{W}} \cdot I(G_{\vshid W \vsdel U}, x)}.
\end{align}
Any consequence for other graph polynomials as a result of this extension would require further studies.

\section*{Acknowledgement}
Many thanks are due to Julian A. Allagan for his suggestions improving the presentation of this paper.

\phantomsection
\addcontentsline{toc}{section}{References}

\printbibliography

\end{document}